\documentclass[11pt]{article}
\usepackage{amssymb,amsmath,amsfonts,amsthm,mathrsfs}
\usepackage{palatino,mathpazo}
\usepackage{fullpage}

\newtheorem{theorem}{Theorem}
\newtheorem{lemma}[theorem]{Lemma}

\newtheorem{claim}[theorem]{Claim}

\newtheorem*{question}{Question}

\newcommand{\Aut}[0]{\operatorname{Aut}}

\def\Stab{\operatorname{Stab}}

\title{\bf Automorphism Groups and Adversarial Vertex Deletions}
\author{Derrick Stolee\\ Department of Mathematics\\ Department of Computer Science\\ University of Nebraska--Lincoln\\ \texttt{s-dstolee1@math.unl.edu}}
\date{\today}

\begin{document}

\maketitle

\begin{abstract}
Any finite group can be encoded as the automorphism group of an unlabeled simple graph.
Recently Hartke, Kolb, Nishikawa, and Stolee (2010) demonstrated a construction that allows any ordered pair of finite groups to be represented as the automorphism group of a graph and a vertex-deleted subgraph.
In this note, we describe a generalized scenario as a game between a player and an adversary:
An adversary provides a list of finite groups and a number of rounds.
The player constructs a graph with automorphism group isomorphic to the first group.
In the following rounds, the adversary selects a group and the player deletes a vertex such that the automorphism group of the corresponding vertex-deleted subgraph is isomorphic to the selected group.
We provide a construction that allows the player to appropriately respond to any  sequence of challenges from the adversary.
\end{abstract}

\vspace{2em}

Automorphisms of graphs are incredibly unstable.
The slightest perturbation of the graph can greatly change the automorphism group.
In this note, we show there exist graphs whose automorphism groups can change dramatically under certain sequences of vertex deletions.
We consider undirected, unlabeled, and simple graphs, denoted $F$, $G$, or $H$, and finite groups, denoted $\Gamma$.
The automorphism group of a graph $G$ is denoted $\Aut(G)$.

Frucht \cite{Fru39} proved that graphs have the ability to encode the structure of any finite group.

\begin{theorem}[Frucht \cite{Fru39}]
	Let $\Gamma$ be a finite group. 
	There exists a graph $G$ with $\Aut(G) \cong \Gamma$.
\end{theorem}

Hartke, Kolb, Nishikawa, and Stolee \cite{HKNS} proved that any ordered pair of finite groups can be represented by a graph and a vertex-deleted subgraph.
Their work was motivated by consequences to the Reconstruction Conjecture (see Bondy \cite{Bondy}) and isomorph-free generation (see McKay \cite{McKay}).

\begin{theorem}[Hartke, Kolb, Nishikawa, Stolee \cite{HKNS}]\label{theorem:vertexall}
	Let $\Gamma_0$ and $\Gamma_1$ be finite groups.
	There exists a graph $G$ and a vertex $v \in V(G)$ such that $\Aut(G) \cong \Gamma_0$ and $\Aut(G-v) \cong \Gamma_1$.
\end{theorem}

There are two natural extensions of this process to a sequence $\Gamma_0, \Gamma_1, \dots, \Gamma_k$ of finite groups using two types of vertex deletions: single deletions or iterated deletions.

\begin{question}\label{thm:generalized1}
	Let $\Gamma_0, \Gamma_1,\dots, \Gamma_k$ be finite groups.
	Does there exist a graph $G$ with vertices $v_1, \dots, v_k \in V(G)$ such that $\Aut(G) \cong \Gamma_0$ and for all $i \in \{1,\dots, k\}$,
	\begin{enumerate}
	\item (Single Deletions) $\Aut(G - v_i) \cong \Gamma_i$?
	\item (Iterated Deletions) $\Aut(G - v_1 - \cdots - v_i) \cong \Gamma_i$?
	\end{enumerate}
\end{question}	

In fact, both of these types of deletions can be combined in an even more general situation, posed as the \emph{vertex deletion game} between a player and an adversary:

\begin{quote}
	\begin{center}
		\bf The Vertex Deletion Game
	\end{center}
	
	{\bf Round 0:} 

		{\it Adversary}: Selects finite groups $\Gamma_0$, $\Gamma_1$, $\dots$, $\Gamma_k$, and a number $\ell \geq 1$.
		
		{\it Player}: Constructs a graph $G_0$ with $\Aut(G_0) \cong \Gamma_0$.
		
	{\bf Round $j$:}\qquad ($1 \leq j \leq \ell$)
	
		{\it Adversary}: Selects a group $\Gamma_{i_j} \in \{\Gamma_1,\dots,\Gamma_k\}$. 
		
		{\it Player}: Selects a vertex $v_j \in V(G_{j-1})$, defines $G_j = G_{j-1} - v_j$, and asserts $\Aut(G_j) \cong \Gamma_{i_j}$.
\end{quote}

Note that this game generalizes both single deletions (play the game with $\ell = 1$) and iterated deletions (play the game with $\ell = k$, and the adversary selects $\Gamma_{i_j} = \Gamma_j$ for all $j \in \{1,\dots, k\}$). 
By carefully constructing $G_0$, the player can survive $\ell$ rounds against the adversary.

\begin{theorem}[Adversarial Iterated Deletions]\label{thm:generalized3}
	Suppose the adversary selects $\Gamma_0$, $\Gamma_1$, $\dots$, $\Gamma_k$ as finite groups and integer $\ell \geq 1$ in Round 0.
	The player can construct a graph $G_0$ with $\Aut(\Gamma_0)$ so that the assertions $\Aut(G_j) \cong \Gamma_{i_j}$ hold for all $\ell$ remaining rounds.
\end{theorem}	

Instead of using the vertex deletion game, there is an equivalent statement of the previous theorem using a sequence of alternating quantifiers.

\begin{theorem}[Adversarial Iterated Deletions; alternate form]
	For all numbers $k, \ell \geq 1$ and finite groups $\Gamma_0$, $\Gamma_1$, $\dots$, $\Gamma_k$, there exists a graph $G_0$ such that $\Aut(G_0) \cong \Gamma_0$ and
	\[
		\forall i_1\, \exists v_1\, \forall i_2\, \exists v_2\, \cdots \forall i_\ell\, \exists v_\ell\, \forall j,\, \Aut(G_0 - v_1 - \cdots - v_j) \cong \Gamma_{i_j},
	\]
	where the domain of $j$ is $\{1,\dots, \ell\}$, the domain of each $i_j$ is $\{1,\dots, k\}$, and the domain of each $v_j$ is $V(G_0) \setminus \{v_1,\dots, v_{j-1}\}$.
\end{theorem}

A group is \emph{trivial} if it consists only of the identity element.
For a graph $G$ and vertex $v \in V(G)$, the \emph{stabilizer} of $v$ in $G$, denoted $\Stab_G(v)$,
	is the subgroup of $\Aut(G)$ given by permutations $\tau$ where $\tau(v) = v$.

Our starting point is the following lemma from \cite{HKNS}.
	
\begin{lemma}[Hartke, Kolb, Nishikawa, Stolee {\cite[Lemma 2.2]{HKNS}}]\label{lemma:stabilizedgamma}
	For any finite group $\Gamma$, there is a connected graph $G$ 
		and a vertex $v \in V(G)$
		where $\Aut(G) \cong \Gamma$ and 
		$\Stab_G(v)$ is trivial.
\end{lemma}

We now describe a gadget which will be used to build the full construction for Theorem \ref{thm:generalized3}.

\begin{lemma}\label{lemma:trivialtoall}
	Let $\Gamma$ be a finite group. 
	There exists a graph $H$ and two vertices $x, y  \in V(H)$
		so that
		$\Aut(H)$ is trivial, 
		$H-x$ is connected, 
		$\Aut(H-x) \cong \Gamma$, 
		and $\Stab_{H-x}(y)$ is trivial.
\end{lemma}

\begin{proof}
	By Lemma \ref{lemma:stabilizedgamma}, there exists a graph $G$ and a vertex $y \in V(G)$ so that $\Aut(G) \cong \Gamma$ and $\Stab_G(y)$ is trivial.
	We shall add vertices to $G$ to form $H$ with trivial automorphism group and a vertex $x \in V(H)$ 	so that the automorphisms of $H - x$ are extensions of automorphisms of $G$ and hence $\Stab_{H-x}(y)$ is trivial.
	
	Let $n = |V(G)|$, $t = \lceil\log n\rceil$, and order the vertices of $G$ as $v_1, \dots, v_n$.
	Label the vertices of a path of order $t+1$ as $u_0, u_1, \dots, u_t$.
	For every vertex $v_j \in V(G)$, let $u_0^{(v_j)}, u_1^{(v_j)}, \dots, u_t^{(v_j)}$ be a copy of this path and identify $u_0^{(v_j)}$ and $v_j$.
	
	Finally, add a vertex $x$ which is adjacent to $v_j$ for all $j \in \{1,\dots, n\}$
		and adjacent to $u_i^{(v_j)}$ if and only if the $i$th bit of the binary expansion of $j$ is equal to $1$.
	Call the resulting graph $H$.
	
	The vertex $x$ is the only vertex of degree at least $2n$, so it is stabilized under automorphisms of $H$.
	However, every vertex $v_j$ in $V(G)$ is identified by which vertices in the path $u_1^{(v_j)},\dots, u_t^{(v_j)}$ are adjacent to $x$.
	This stabilizes every vertex in $V(G)$ and hence every vertex of $H$.
	Thus, $\Aut(H)$ is trivial and $H-x$ is connected.

	The vertex-deleted graph $H-x$ is given by the graph $G$ with a path of order $n+1$ attached to every vertex.
	Observe that $V(G)$ is set-wise stabilized under automorphisms of $H-x$ and the automorphisms of $G$ naturally extend to automorphisms of $H - x$.
	Thus, $\Aut(H-x) \cong \Aut(G) \cong \Gamma$ and $\Stab_{H-x}(y)$ is trivial.
\end{proof}

We are now sufficiently prepared to prove the main theorem.
The gadget from Lemma \ref{lemma:trivialtoall} has two purposes:
\begin{enumerate}
	\item ``Reveal'' symmetry: When $x$ is deleted, the automorphism group $\Gamma$ is revealed.
	\item ``Remove'' symmetry: When $y$ is stabilized within $H - x$, all non-trivial automorphisms of $H-x$ are removed.
\end{enumerate}
Our construction for the graph $G_0$ carefully places many copies of this gadget in such a way that the player has access to a ``revealing'' vertex ($x$) that simultaneously stabilizes the ``removing" vertex ($y$) in the previous gadget.
Therefore, we have a sequence of deletions which remove all previous symmetry and reveal only the requested symmetry.

\begin{proof}[Proof of Theorem \ref{thm:generalized3}]
	Note that the case $k = \ell = 1$ holds by Theorem \ref{theorem:vertexall}.
	We assume that the groups $\Gamma_1, \dots, \Gamma_k$ are distinct with respect to isomorphism.
		
	By Lemma \ref{lemma:trivialtoall}, for every $i \in \{0,1,\dots,k\}$ there is a graph $H_i$ with vertices $x_i, y_i \in V(H_i)$ such that $\Aut(H_i)$ is trivial, $\Aut(H_i - x_i) \cong \Gamma_i$, and $\Stab_{H_i-x_i}(y_i)$ is trivial.
	For all $i \in \{0,\dots,k\}$, let $O_i$ be the orbit of $y_i$ in $H_i - x_i$.
	Since the groups $\Gamma_1, \dots, \Gamma_k$ are pairwise non-isomorphic, the graphs $H_1,\dots, H_k$ are pairwise non-isomorphic, as are the graphs $H_1 - x_1, \dots, H_k - x_k$.
	
	We construct the graph $G_0$ by building graphs $F_0, F_1, \dots, F_\ell$ iteratively.
	Let $F_0 = H_0 - x_0$ and $U_0 = O_0$.
	For all $j \in \{0,\dots, \ell-1\}$, we will build $F_{j+1}$ from $F_j$.

	Consider each $j \geq 0$.
	For all vertices $v \in U_j$ and $i \in \{1,\dots,k\}$, create a copy $H_i^{(j+1,v)}$ of $H_i$ and add edges from $v$ to each vertex of $H_i^{(j+1,v)}$.
	Let $x_{i}^{(j+1,v)}$ and $y_{i}^{(j+1,v)}$ denote the copies of $x_{i}$ and $y_{i}$ in $H_{i}^{(j+1,v)}$.
	Let $O_i^{(j+1,v)}$ be the copy of $O_i$ within $H_i^{(j+1,v)}$ and define $U_{j+1} = \displaystyle\cup_{v \in U_j} \displaystyle\cup_{i=1}^k O_i^{(j+1,v)}$.

	Now, add a path $a_0, \dots, a_\ell$ to $F_\ell$, and add edges such that $a_j$ is adjacent to all vertices in $V(F_j) \setminus V(F_{j-1})$ (the vertex $a_0$ is adjacent to $V(F_0)$).
	Call the resulting graph $G_0$.
	
	Observe that each vertex $a_j$ is distinguished by its degree (the sizes of the sets $V(F_j)$ increase geometrically as $j$ increases).
	Therefore, every set $U_j$ is identified in $G_0$, and therefore set-wise stabilized.
	In particular, $U_0$ is set-wise stabilized, so any automorphisms of $G_0$ must preserve $U_0$ and similarly $H_0-x_0$.

	\begin{claim}\label{claim:induction}
		Fix $j \geq 0$ and $X \subseteq V(G_0)$. 
		If $X \subseteq V(F_j)$, then $\Aut(G_0 - X) \cong \Aut(F_j - X)$.
	\end{claim}
	
	\begin{proof}[Proof of Claim \ref{claim:induction}:]
		Since the vertices $a_0, \dots, a_\ell$ are not included in $V(F_j)$, they remain stabilized in $G_0 - X$.
		Therefore, the sets $V(F_{j'+1}) \setminus V(F_{j'})$ are set-wise stabilized in $\Aut(G_0 - X)$ for all $j' \in \{j,\dots,\ell-1\}$.
		We show the natural map from $\Aut(F_{j'+1}-X)$ to $\Aut(F_{j'}-X)$ is a bijection for all $j' \in \{j,\dots, \ell-1\}$, implying the natural map between $\Aut(F_{\ell}-X)$ and $\Aut(F_j-X)$ is a bijection.
		
		Every vertex $u \in V(F_{j'+1}-X) \setminus V(F_{j'})$ is contained in $H_i^{(j'+1,v)}$ for some vertex $v \in U_{j'}\setminus X$ and $i \in \{1,\dots,k\}$.
		Since $V(H_i^{(j'+1,v)}) \cap X = \emptyset$, this subgraph $H_i^{(j'+1,v)}$ has no non-trivial automorphisms.
		Therefore, for every automorphism $\sigma$ of $F_{j'} - X$, there is exactly one isomorphism of $F_{j'+1}-X$ that extends $\sigma$ and maps $V(H_{i}^{(j'+1,v)})$ to $V(H_i^{(j'+1, \sigma(v))})$.
		Hence, the action of an automorphism on each vertex $u \in V(F_{j'+1}-X)\setminus V(F_{j'})$ is determined exactly by the action of the automorphism on the vertices within $V(F_{j'}-X)$.
		Hence, the restriction map from $\Aut(F_{j'+1}-X)$ to $\Aut(F_{j'}-X)$ is a bijection, proving the claim.
	\end{proof}
	
	When $X = \emptyset$, the automorphism group of the subgraph $F_0$ determines the automorphism group of $G_0 - X$.
	Since $F_0 \cong H_0-x_0$, then $\Aut(G_0) \cong \Gamma_0$.
	
	We now play the vertex deletion game as the player, and we shall always select the vertex $v_j$ in Round $j$ from $V(F_j) \setminus V(F_{j-1})$, for all $j \in \{1,\dots,\ell\}$.
	First, let $u_0$ be the vertex $y_0$ within $H_0-x_0$.
	
	If the adversary selects $\Gamma_{i_j}$ in Round $j$, we select $v_j$ as the copy of $x_j$ in $H_{i_j}^{(j, u_{j-1})}$.
	Then, we define $u_j$ as the copy of $y_{i_j}$ in $H_{i_j}^{(j, u_{j-1})}$.
	Observe that if $X_j = \{v_1,\dots,v_j\}$, then $X_j \subseteq V(F_j)$.
	Also, each vertex $u_{j'}$, where $j' < j$, is stabilized within $U_{j'}$, since the neighborhood of $u_{j'}$ in $U_{j'+1}$ (isomorphic to $H_{i_{j'}} - x_{i_{j'}}$) is not isomorphic to the similarly restricted neighborhood of any other vertex in $U_{j'} \setminus \{u_{j'}\}$ (which is isomorphic to $H_{i_{j'}}$).
	 In particular, this stabilizes $u_{j'}$ within $H_{i_{j'}}^{(j', u_{j'-1})}$, so no non-trivial automorphisms exists in $F_{j-1} - X_{j}$.
	 However, the subgraph $H_{i_j}^{(j,u_{j-1})}$ has automorphism group isomorphic to $\Gamma_{i_j}$, and each of these automorphisms extend to a unique automorphism of $G_0 - X_j$.
	 Therefore, $\Aut(G_j) = \Aut(G_0-X_j) \cong \Gamma_{i_j}$.
\end{proof}

The construction given in the above proof requires a large number of vertices 
	and vertices of high degree.
While the gadget given by Lemma \ref{lemma:trivialtoall} can be built using $O(|\Gamma|\log^2 |\Gamma|\log\log |\Gamma|)$ vertices\footnote{The construction of Lemma \ref{lemma:stabilizedgamma} from \cite{HKNS} has order $O(|\Gamma|^4)$, but can  be replaced by a construction of Sabidussi \cite{Sabidussi} with $O(|\Gamma|\log|\Gamma|\log\log|\Gamma|)$ vertices. Then, carefully applying the construction of Lemma \ref{lemma:trivialtoall} to Sabidussi's construction, the number of vertices is increased by a multiplicative factor of $O(\log |\Gamma|)$.}, Babai \cite{Babai} proved that for every finite group $\Gamma$ there is a graph $G$ with $\Aut(G) \cong \Gamma$ and $|V(G)| \leq 3|\Gamma|$.
Can graphs with $O(|\Gamma|)$ vertices be used to satisfy Lemma \ref{lemma:trivialtoall}?
Also, the constructions used here contain vertices of high degree. 
Does there exist a constant $D$ so that Theorem \ref{thm:generalized3} is satisfied with the maximum degree of $G_0$ at most $D$?

%

\bibliographystyle{plain}

\end{document}